\documentclass[10pt]{amsart}

\usepackage{amsthm,amsmath,amssymb,amscd}
\usepackage{amssymb}
\usepackage{graphics}

\usepackage{syntonly}
\usepackage{pstricks,subfigure,epsfig}
\usepackage{amsmath,amsfonts,amssymb,amsthm,euscript,upgreek}
\usepackage{enumerate}
\usepackage{multirow}
\usepackage{appendix}

\newtheorem{theo}{Theorem}[section]

\newtheorem{remar}[theo]{Remark}
\newtheorem{prop}[theo]{Proposition}
\newtheorem{corol}[theo]{Corollary}
\newtheorem{lemma}[theo]{Lemma}

\newtheorem{assumption}{Assumption}

\newcommand{\complex}{\mathbb{C}}

\newcommand{\natur}{\mathbb{N}}

\newcommand{\K}{K\"{a}hler}

\newcommand{\R}{\Bbb{R}}

\newcommand{\C}{\Bbb{C}}
\newcommand{\T}{\Bbb{T}}

\newcommand{\Z}{\Bbb{Z}}

\newcommand{\Proj}{\Bbb{P}}
\newcommand{\ep}{\epsilon}

\newcommand{\Ric}{\operatorname{Ric}}

\newcommand{\PGL}{\operatorname{PGL}}

\newcommand{\e}{\varepsilon}
\newcommand{\del}{\partial}

\theoremstyle{plain}

\def\Aut{\mathop{\hbox{Aut}}}
\def\Lie{\mathop{\hbox{Lie}}}

\begin{document}

\title[On homothetic balanced metrics] {On homothetic balanced metrics}

\author{Claudio Arezzo}
\address{Abdus Salam International Center for Theoretical Physics \\
                  Strada Costiera 11 \\
         Trieste (Italy) and Dipartimento di Matematica \\
         Universit\`a di Parma \\
         Parco Area delle Scienze~53/A  \\
         Parma (Italy)}
\email{arezzo@ictp.it}

\author{Andrea Loi}
\address{Dipartimento di Matematica \\
         Universit\`a di Cagliari}
         \email{loi@unica.it}

\author{Fabio Zuddas}
\address{Dipartimento di Matematica \\
          Parco Area delle Scienze 53/A \\
         Parma (Italy)}
\email{fabio.zuddas@unipr.it}

\thanks{
The authors are supported  by the F.I.R.B. 2008 Project \lq\lq Geometria Differenziale Complessa e Dinamica Olomorfa''}

.
\date{\today}

\begin{abstract}
In this paper we study the set of  balanced metrics (in Donaldson's terminology
\cite{do}) on a  compact complex manifold $M$ which are  homothetic to a given balanced one. This question is related to various properties of the Tian-Yau-Zelditch approximation theorem for \K\  metrics.
We prove that this set is finite when $M$ admits a   non-positive
\K--Einstein metric, in the case of non-homogenous   toric \K-Einstein manifolds  of dimension $\leq 4$
and in the case of the constant scalar curvature metrics
found in \cite{AP} and \cite{AP2}.
\end{abstract}

\maketitle

\section{Introduction}

A fundamental result of Tian (\cite{ti0}) states that any \K\ metric on a  compact manifold is the limit of projectively induced metrics. Moreover a quantitative and refined version of this result, due to Lu (\cite{lu}) and Zelditch (\cite{ze}), gives an asymptotic expansion (since then called Tian-Yau-Zelditch expansion) for a suitably chosen sequence of projective metrics. The nature of the coefficients of this expansion is a challenging and intriguing question, in some sense resembling, in a complex form, better known similar problems in Riemannian geometry such as that of isospectral manifolds. This circle of questions turns out to be relevant also in problems coming from the theory of geometric quantization and in the existence problem of constant scalar curvature \K\ metrics (\cite{do}) via the notion of balanced metrics introduced by Donaldson.

Of course it is of particular interest to characterize those \K\ manifolds whose coefficients of the associated  TYZ expansion are constants. We will observe in Section $2$ that this property is implied by having infinitely many proportional, here called homothetic, balanced metrics (a property already studied in the context of geometric quantization as recalled below) and it is related to another natural question about the characterization of the projectively induced metrics.

These properties are those central in this paper.

To enter more in detail, fix a positive line bundle $L$ over a compact complex manifold $M$ and  denote by ${\mathcal B}(L)$  the set of balanced metrics  on $M$ which are polarized either  with respect to  $L$ or some of its tensor powers,  namely,  $g_B\in {\mathcal B}(L)$ iff $g_B$ is balanced and  there exists a non-negative integer  $m_0$ such that $\omega_B$, the \K\ form associated to $g_B$, belongs to $c_1(L^{m_0})$.
For a fixed $g_B\in {\mathcal B}(L)$ consider the set of all balanced metrics homothetic to $g_B$, namely the set
${\mathcal B}_{g_B}=\{mg_B \  \mbox {is balanced}  \ | \ m\in\natur^+  \}$.

Obviously,  ${\mathcal B}_{g_B}\subset {\mathcal B}(L)$ for each  $g_B\in {\mathcal B}(L)$.
Notice that   two balanced metrics in  ${\mathcal B}(L)$ are isometric if and only if
their associated \K\ forms  are cohomologous (see \cite{arlcomm} or  Theorem  \ref{thmarloi} below)  and hence the cardinality of  ${\mathcal B}_{g_B}$ is a cohomological invariant.
For this reason we consider the quotient, denoted by ${\mathcal B}_c(L)$,  of ${\mathcal B}(L)$ by the equivalence relation which identifies two
balanced metrics if they belong to the same cohomology class.
Observe that if the polarized manifold $(M, L)$ is asymptotically Chow polystable then by a fundamental result of S. Zhang \cite{zha} (see the next section) the cardinality of  ${\mathcal B}_c(L)$ is infinite.
Moreover,  by a result in \cite{vezu} there exist examples  where  $(M, L^m)$ is not Chow polystable (even of constant scalar curvature) for $m$ large enough and hence, in this case, the  cardinality of   ${\mathcal B}_c(L)$, and hence that of
${\mathcal B}_{g_B}$, is finite (possibly zero).
On the other hand,  it is not hard to verify that any homogeneous integral  \K\ metric $g$ on a simply-connected homogeneous compact complex manifold is
such that  $mg$ is balanced for all  sufficiently large non-negative integers and hence the set ${\mathcal B}_{g_B}$, $g=g_B$,   (and a fortiori ${\mathcal B}_c(L)$) has infinite cardinality (see \cite{arlcomm}).
More generally, given a \K\ metric $g_B$ polarized with respect to  $L$ one tries to understand when $mg_B$ is balanced for all non-negative integers (and so our set ${\mathcal B}_{g_B}$ has infinite cardinality in this case).
If this happens the  corresponding geometric quantization is called regular. Regular quantizations play  a fundamental role in the theory of Berezin quantization by deformation developed by M. Cahen, S. Gutt and J. Rawnsley
in \cite{cgr1}, \cite{cgr2}, \cite{cgr3}, \cite{cgr4}.  A complete classification of regular quantizations is still missing
(we refer to \cite{arlquant}, \cite{arlcomm} and \cite{loibalcov}  for more details).
We believe that the \K\ manifolds which admits a   regular quantization or, more generally, those for which  ${\mathcal B}_{g_B}$ consist of infinite elements, are in some sense special. More precisely,
we believe the validity of the following

\vskip 0.3cm

\noindent
{\bf Conjecture:} Let $(M, L)$ be a polarized manifold. If there exists  a balanced metric $g_B\in {\mathcal B}(L)$
such that $\sharp {\mathcal B}_{g_B}=\infty$ then $(M, g_B)$ is a homogeneous \K\ manifold.

\vskip 0.3cm

Our main results are the following three results (see Sections \ref{secmainteor1}, \ref{secmainteor2} and \ref{secmainteor3} below   for details).

The first   observation, since \K--Einstein metrics with  non-positive scalar curvature are never projectively induced, is then the following:

\begin{prop}\label{mainteor1}
Let $L$ be a polarization of a compact  \K--Einstein manifold $(M, g)$ with  non-positive scalar curvature.
Then ${\mathcal B}_c(L)$  consists of  infinitely many  balanced metrics such that  for each   $g_B\in {\mathcal B}(L)$,  the set
${\mathcal B_{g_B}}$ is finite.
\end{prop}

On the other hand projectively induced \K--Einstein manifolds $(M, g)$ with positive scalar curvature do exist, and it has been repeatedly claimed that only homogeneous manifolds have this property. Unfortunately all these proofs contain fatal errors. The proof of Theorem \ref{mainteor2} is based on the fact that a  \K--Einstein metric on toric manifolds of dimension $\leq 4$ are not projectively induced.  
This is the first class of \K--Einstein metrics on compact complex manifolds $M$, with $c_1(M)>0$ and large group of isometries, for which we can prove such property.

\begin{theo}\label{mainteor2}
Let $g$ be a  \K\--Einstein metric on a toric manifold $M$ of dimension $\leq 4$ and let $L=K^*$ be the anticanonical bundle over $M$.
Then ${\mathcal B}_c(L)$  consists of infinitely many balanced metrics. Moreover,   there exists   $g_B\in {\mathcal B}(L)$
such that   ${\mathcal B_{g_B}}$ is infinite if and only if $M$ is either a projective space or a product of projective spaces.
\end{theo}

Passing from \K -Einstein to constant scalar curvature (cscK) metrics we prove that in what is at present
the greatest source of examples, namely the blow up gluing procedure developed in
\cite{AP} and \cite{AP2}, the second coefficient of the TYZ asymptotic expansion is never constant, thanks to some special properties of the LeBrun-Simanca model for the gluing procedure. This implies the following:

\begin{theo}\label{mainteor3}
Let $g$ be a cscK metric on a compact complex manifold and let
$g_{\e}$, $\ep >0$,  be a family of cscK metrics constructed as in
\cite{AP} and \cite{AP2} on the blow-up $\tilde M = Bl_{p_1,\dots ,p_k}M$ of $M$ at
the points $p_1, \dots , p_k$ of $M$.
Let  $\e$ be  a sufficiently small rational number, say $\e = \frac{p}{q}$,
and let  $L_{\e} \rightarrow \tilde M$ be a polarization for the \K\ class
of the metric $qg_{\e}$. Then,  for each   $g_B\in {\mathcal B}(L_{\e})$,  the set
${\mathcal B_{g_B}}$ is finite.
\end{theo}

The authors believe that the computation of the cardinality of ${\mathcal B}_{g_B}$ or the
proof of the  existence  of an upper bound  of this cardinality when $g_B$ is varying in ${\mathcal B}(L)$ is a very hard   and intriguing  problem which could shed some  light to the understanding of balanced metrics and of the stability of the polarized manifold $(M, L)$.

\vskip 0.3cm
 The paper is organized as follows. In the next section we describe the link between balanced and projectively induced \K\ metrics, we
 recall  the TYZ expansion and we use it
to prove two lemmata needed in the proof of the main results.
Sections \ref{secmainteor1}, \ref{secmainteor2} and \ref{secmainteor3} are dedicated to the proofs of
Proposition \ref{mainteor1},  Theorem \ref{mainteor2} and Theorem \ref{mainteor3} respectively.

\vskip 0.3cm

\noindent
{\bf Acknowledgements:} We wish to thank Prof. Zhiqin Lu  for various interesting and stimulating discussions.

\section{Balanced and projectively induced metrics}

Let $g$ be a \K\ metric on a compact complex manifold $M$.
In the quantum mechanics terminology $(M, g)$ is said to be  {\em quantizable} if the \K\ form $\omega$
associated to $g$ is integral, i.e. there exists a holomorphic line bundle $L$ over $M$, called the {\em quantum line bundle},  whose first Chern class equals the second  De-Rham cohomology class of $\omega$, i.e.  $c_1(L)=[\omega]_{dR}$.
By Kodaira's theory this is equivalent to say  that $M$ is a projective algebraic manifold and $L$ is a positive (or ample) line bundle over $M$.
In algebraic-geometric terms  $L$  is said to be a {\em polarization} of $M$, $g$   a {\em polarized} metric and
the  pair $(M, L)$  a {\em polarized manifold}.
Fix a polarization   $L$ over $M$.
Then there exists
an hermitian metric  $h$  on
$L$, defined up to the multiplication  with a postiive constant,
such that its Ricci curvature $\Ric
(h)=\omega$ \footnote{$\Ric (h)$ is the two--form on $M$ whose
local expression is given by
$\Ric (h)=-\frac{i}{2}
\partial\bar\partial\log h(\sigma (x), \sigma (x)),$
for a trivializing holomorphic section $\sigma :U\rightarrow
L\setminus\{0\}$.}.
The  pair $(L, h)$ is called a {\em geometric
quantization} of the  K\"{a}hler manifold $(M, \omega)$.
Let  $s_0, \dots , s_{N}$  be an orthonormal basis of $H^0(L)$ (the space  of  global holomorphic sections  of $L$) with respect to the scalar product
$$\langle s, t \rangle  =\int_Mh(s(x), t(x))\frac{\omega^n(x)}{n!}, \ s, t\in H^0(L).$$
Consider the  non-negative smooth  function
$T_{g}$ on $M$ given by:
\begin{equation}\label{Tmo}
T_{g} (x) =\sum_{j=0}^{N}h(s_j(x), s_j(x)).
\end{equation}
As suggested by the notation this function depends only on the
\K\ form  metric $g$ and not on the orthonormal basis chosen.

The function $T_{g}$ has appeared in the literature under different
names. The earliest one was probably the $\eta$-function of J.
Rawnsley \cite{ra1} (later renamed to $\theta$ function in
\cite{cgr1}), defined for arbitrary (not necessarily compact) K\"{a}hler manifolds,
followed by the {\em distortion function } of G. R. Kempf \cite{ke} and S. Ji
\cite{ji}, for the special case of Abelian varieties and of S. Zhang
\cite{zha} for complex projective varieties.
The metrics for which
$T_{g}$ is constant were called {\em critical} in \cite{zha} and {\em
balanced}  in \cite{do}.

The two   fundamental  results about existence and uniqueness of balanced metrics  are summarized in the following two theorems.
\begin{theo}\label{thmdo} (S. Donaldson \cite{do})
Let $(L, h)$ be a geometric quantization
of a compact \K\ manifold $(M, \omega)$
such that the polarized metric  $g$ whose associated \K\ form  $\omega$ has constant scalar curvature.
Assume that $\frac{\Aut (M, L)}{{\C}^*}$ \footnote{{$\frac{\Aut(M, L)}{{\C}^*}$
denotes the group biholomorphisms of
$M$ which lift to holomorphic bundles maps
$L\rightarrow L$ modulo the trivial automorphism
group ${\complex}^*$}. } is discrete.
Then, for all  sufficiently large   integers  $m$ ,  there exists a unique balanced metric $\tilde g_m$ on $M$, with polarization
$L^m=L^{\otimes m}$, such that
$\frac{\tilde g_m}{m}$ $C^{\infty}$-converges to $g$.
Moreover if $\tilde g_m$ is a sequence of balanced metrics on $M$ with  $\tilde\omega_m\in c_1(L^m)$
such that
$\frac{\tilde g_m}{m}$ $C^{\infty}$-converges to a metric $g$ then $g$ has constant scalar curvature.
\end{theo}

Besides the uniqueness part which is recalled below, Mabuchi (\cite{mab}) extended the above theorem to polarized manifolds with nontrivial automorphisms under certain conditions. Moreover, a beautiful dynamical version of the above theorem has been given by J. Fine in \cite{fi}.

From the GIT (geometric invariant theory) point of view given a polarized manifold  $(M, L)$, with $L$ very ample,   there exists a  balanced metric  $g$ whose associated \K\ form is in the class  of $c_1(L)$  if and only if   $(M, L)$ is Chow polystable (see \cite{zha} for a proof).  Since the Chow polystability is equivalent to the Chow stability  when $\frac{\Aut (M, L)}{{\C}^*}$  is discrete,
Theorem \ref{thmdo} can be equivalently stated by saying that given a  polarized manifold $(M, L)$ such that  $\frac{\Aut (M, L)}{{\C}^*}$ is discrete  and  $M$ admits a  constant scalar curvature metric in the class $c_1(L)$ then $(M, L)$ is asymptotically Chow stable
 (i.e. Chow stable for all $m$ sufficiently large). Notice  that the assumption on the automorphism group in Theorem \ref{thmdo}  cannot  be dropped entirely.
Indeed, from the point of view of the existence of balanced metrics  the recent results of  \cite{osy} and of A. Della Vedova and the third author \cite{vezu}  show  that there exist  a large class of  polarized manifolds  $(M, L)$ such that $M$ admits a  constant scalar curvature metric in the class $c_1(L)$  and such that   $(M, L^m)$ is not polystable, for all $m$ sufficiently large. Regarding the uniqueness of balanced metrics   the first  and the second  author   \cite{arlcomm} have  shown the following:
\begin{theo}\label{thmarloi}
Let  $g$ and $\tilde g$ be   two balanced metrics  whose associated \K\ forms are cohomologous  Then $g$ and $\tilde g$
are isometric, i.e.  there exists $F\in \Aut (M)$ such that $F^*\tilde g =g$.
\end{theo}

For  a polarization   $L$ over  $(M, g)$ and every non-negative integer $m\geq 1$ let us
consider the Kempf distortion function associated to $mg$, i.e.
\begin{equation}\label{Tm}
T_{mg} (x) =\sum_{j=0}^{d_m}h_m(s_j(x), s_j(x)),
\end{equation}
where  $h_m$ is an hermitian metric   on
$L^m$ such that $\Ric(h_m)=m\omega$
and   $s_0, \dots , s_{d_m}$, $d_m+1=\dim H^0(L^m)$, is  an orthonormal basis of $H^0(L^m)$
with respect to the $L^2$-scalar product
$$\langle s, t \rangle_m =\int_Mh_m(s(x), t(x))\frac{\omega^n(x)}{n!}, \ s, t\in H^0(L^m).$$
(In the  quantum geometric context
$m^{-1}$ plays the role of Planck's constant, see e.g.
\cite{arlquant}).

One can give a
quantum-geometric interpretation of $T_{mg}$   as follows.
Take $m$ sufficiently large such  that for each point $x\in M$
there exists $s\in H^0(L^m)$ non-vanishing at $x$ (such an $m$ exists by standard algebraic geometry methods
and corresponds to the free-based point condition in Kodaira's theory, see e.g. \cite{kn}).  Consider the
so called  {\em coherent states map}, namely the
holomorphic map of $M$ into the complex projective space
${\complex}P^{d_m}$ given by:
\begin{equation}\label{psiglob}
\varphi_m :M\rightarrow {\complex}P^{d_m}, \
x\mapsto [s_0(x): \dots :s_{d_m}(x)].
\end{equation}
One can prove (see, e.g.  \cite{arlcomm}) that
\begin{equation}\label{obstr}
\varphi ^*_m\omega_{FS}=m\omega +
\frac{i}{2}\partial\bar\partial\log T_{mg} ,
\end{equation}
where $\omega_{FS}$ is the Fubini--Study form on
${\complex}P^{d_m}$, namely the  \K\ form which in homogeneous
coordinates $[Z_0,\dots, Z_{d_m}]$ reads as \linebreak
$\omega_{FS}=\frac{i}{2}\partial\bar\partial\log \sum_{j=0}^{d_m}
|Z_j|^2$.
Since the equation  $\partial\bar\partial f=0$
implies that $f$ is constant,  it follows  by (\ref{obstr})
that $m g$  is balanced if and only if
it is projectively induced via the coherent states map.
Recall that   a polarized   K\"{a}hler metric $g$ on a complex manifold $M$
with polarization $L$ is
 {\em projectively induced} if there exists a basis $t_0, \dots  .t_N$  of   $H^0(L)$
 such that the holomorphic map  $\psi : M\rightarrow
{\complex}P^N, x\mapsto [t_0(x): \cdots : t_N(x)]$ induced by this basis, satisfies $\psi^*(g_{FS})=g$
(the author is referred to the seminal paper of E. Calabi \cite{ca} for more detalis on
the subject).
Notice that  there is a large class of  projectively induced \K\ metrics
which are not balanced. Indeed by Theorem \ref{thmarloi}  the set of balanced  metrics
on  a fixed cohomology class   is either empty or in bijection with the automorphism group $\Aut (M)$ of the manifold, while
by Calabi's rigidity theorem (see \cite{ca}) the set of projectively induced metrics in the same class are in bijection
with $ \PGL  (N+1)/U(N+1)\times \Aut (M)$ ($N=\dim H^0(L)-1$).
\footnote{
To have an explicit example take the metric  $g=\psi^*g_{FS}$ on $\C P^1$,
where $\psi:\C P^1\rightarrow \C P^2$ is the Veronese embedding, i.e. $\psi ([z_0, z_1])=[z_0^2, z_0z_1, z_1^2]$.
Then $g$ is  a projectively induced \K\   metric on $\C P^1$  polarized with respect to $O(2)$ and $g$ is not balanced.}

Not all K\"{a}hler metrics are balanced or projectively induced.
Nevertheless, G. Tian \cite{ti0} and  W. D. Ruan \cite{ru} solved  a conjecture posed by Yau
by showing that
$\frac{\varphi_m ^*(g_{FS})}{m}$ $C^{\infty}$-converges to $g$.
 In other words, any polarized
metric on a compact complex manifold is the $C^{\infty}$-limit of
(normalized) projectively induced K\"{a}hler metrics. S. Zelditch
\cite{ze} generalized the Tian--Ruan theorem by proving a complete
asymptotic expansion in the $C^\infty$ category, namely
\begin{equation}\label{asymptoticZ}
T_{mg}(x) \sim \sum_{j=0}^\infty  a_j(x)m^{n-j},
\end{equation}
where  $a_j(x)$, $j=0,1, \ldots$, are smooth coefficients with $a_0(x)=1$.
More precisely,
for any nonnegative integers $r,k$ the following estimates hold:
\begin{equation}\label{rest}
||T_{mg}(x)-
\sum_{j=0}^{k}a_j(x)m^{n-j}||_{C^r}\leq C_{k, r}m^{n-k-1},
\end{equation}
where $C_{k, r}$ is a constant depending on $k, r$ and on the
K\"{a}hler form $\omega$ and $ || \cdot ||_{C^r}$ denotes  the $C^r$
norm in local coordinates.
 Later on,  Z. Lu \cite{lu} (see also \cite{liulu}),  by means of  Tian's peak section method,
 proved  that each of the coefficients $a_j(x)$ in
(\ref{asymptoticZ}) is a polynomial
of the curvature and its
covariant derivatives at $x$ of the metric $g$ which can be found
 by finitely many algebraic operations.
 Furthermore,  he explicitely computes
$a_j$ for $j\leq 3$,
 i.e. (we omit the expression of $a_3$ since we do not need it in this paper)
\begin{equation}\label{coefflu}
\left\{\begin{array}{l}
a_1(x)=\frac{1}{2}\rho\\
a_2(x)=\frac{1}{3}\Delta\rho
+\frac{1}{24}(|R|^2-4|\Ric |^2+3\rho ^2)\\
\end{array}\right.
\end{equation}
where
$\rho$, $R$, $\Ric$ denote respectively the scalar curvature,
the curvature tensor and the Ricci tensor of $(M, g)$.
The reader is also referred to   \cite{loianal} and
\cite{loismooth}  for a  recursive formula of the $a_j$'s and  an alternative  computations of  $a_j$ for $j\leq 3$ using
Calabi's diastasis function (see also   the recent papers \cite{xu1} and \cite{xu2} for a graph-theoretic
interpretation of this recursive formula).
The expansion
(\ref{asymptoticZ}) is called the  {\em TYZ (Tian--Yau--Zelditch)
expansion}.  Together with Donaldson's moment maps techniques, it  is a key ingredient  in   the proof of Theorems \ref{thmdo} and \ref{thmarloi}
in the Introduction.

We end this section with two lemmata needed in the proofs of our theorems.

\begin{lemma}\label{constcoeff}
Let $g$ be any (not necessarily balanced)  polarized metric on  a compact complex manifold $M$.
Assume that  the set
$${\mathcal B}_{g}=\{mg\  \mbox {is balanced} \ | \ m\in\natur   \}$$
 consists of infinite elements. Then the coefficients $a_j(x)$ of the TYZ expansion of the Kempf distortion function $T_{mg}$  are constants
 for all $j=0, 1, \dots$.
\end{lemma}
\begin{proof}
Assume that there exists an    increasing sequence $\{m_s\}_{s=1, 2, \dots}$  of non-negative integers such that $m_sg$
is a balanced metric, i.e.
$T_{m_sg}(x)=T_{m_s}$ for some positive constants  $T_{m_s}$.
We argue by induction on $j$. We already know that $a_0=1$ is a constant, so assume that the $a_j(x)$'s are constants, say $a_j$,  for $j=0, \dots ,k-1$.
By (\ref{rest}) we have
$$|T_{s, k, n}-a_k(x)m_s^{n-k} |\leq C_{k}m_s^{n-k-1}$$
for some constant $C_k$,
where $T_{s, k, n}$ is the constant (depending on $s, k$ and $n$) equal to $T_{m_s}-\sum_{j=0}^{k-1}a_jm_s^{n-j}$.
Hence $|m_s^{k-n}T_{s, k, n}-a_k(x)|\leq C_km_s^{-1}$ and letting $s\rightarrow\infty$ we get that $m_s^{k-n}T_{s, k, n}$ tends to  $a_k(x)$
which is then forced to be a constant.
\end{proof}

As a simple consequence of the previous lemma we get the following result which shows the validity of the Conjecture in the Introduction
in the one-dimensional case.

\begin{corol}
Let  $(M, L)$ be  a polarized  manifold and   $M$ have complex dimension $1$.
Assume that there exists    $g_B\in {\mathcal B}(L)$ such that
$\sharp {\mathcal B_{g_B}}=\infty$. Then   $M$ is biholomorphic to the the Riemann sphere $\C P^1$.
\end{corol}
\begin{proof}
Assume  $\sharp {\mathcal B_{g_B}}=\infty$ for some  $g_B\in {\mathcal B}(L)$. Then,   by Lemma \ref{constcoeff}, the coefficients $a_j^B$ of the TYZ expansion of  $T_{mg_B}$ are constants. In particular $a_1^B=\rho_B/2$ is constant, where $\rho_B$ is the scalar curvature  of $g_B$ (cfr.  (\ref{coefflu})). On the other hand
the flat metric on an elliptic curve and the hyperbolic metric on a Riemann surface of genus $\geq 2$ cannot be projectively induced
(see \cite{loiherm} for a proof) and hence $M$ is forced to be biholomorphic to  $\C P^1$ (and $g_B$ isometric to an integer multiple of the Fubini--Study metric).
\end{proof}

\begin{lemma}\label{constcoeff2}
Let $g$ be a cscK  metric on a compact complex manifold $M$  polarized with respect to a holomorphic line bundle $L$.
Assume $g$ satisfies one of the following conditions:
\begin{enumerate}
\item
$mg$ is not  projectively induced for all  $m$;
\item
there is  at least a non-constant coefficient $a_{j_0}$, with  $j_0\geq 2$, of  the TYZ expansion (\ref{asymptoticZ}) of the  Kempf distortion function $T_{mg}(x)$.
\end{enumerate}
Then,
for any $g_B\in {\mathcal B}(L)$,
the set  ${\mathcal B}_{g_B}$
consists of finitely many elements.
\end{lemma}
\begin{proof}
Let $g_B\in {\mathcal B}(L)$, this means that $g_B$ is    a balanced metric on $M$ such that  its  associated \K\ form $\omega_B$ belongs to $c_1(L^{m_0})$
for some $m_0$.   Assume by a contradiction that  ${\mathcal B}_{g_B}$ has infinite elements. Then, by Lemma \ref{constcoeff},  the coefficients of the TYZ expansion of $g_B$, denoted by   $a_j^B$,  are constants for all  $j=0, 1, \dots$.
In particular, by the first  of  (\ref{coefflu}),  $g_B$ is cscK. Since $\omega_B$  is cohomologous to $m_0\omega$ and by assumption $g$ has constant scalar curvature it follows by a theorem of  X. X. Chen and G. Tian  \cite{chentian} that there exists an automorphism $F$ of $M$ such that $F^*g_B=m_0g$.
Since $g_B$ is projectively induced and the $a_j^B$'s are constants for all   $j=0, 1, \dots$ we get that: (a) $m_0g$ is projectively induced;
(b) the coefficients $a_j$'s of the TYZ expansion  of $T_{mg}$ are constants for all $j=0, 1, \dots$.. Since  (a) and (b) are in constrast with  (1) and (2) respectively
this yields the desired contradiction and concludes the proof of the lemma.
\end{proof}


It is interesting  pointing out that
there are examples, even in  complex dimension $1$, of
cscK metrics  such that all the coefficients of the associated  TYZ expansion are constant  and
$mg$ is not projectively induced for all non negative integer  $m$.
Take for example a  compact  Riemann surface $\Sigma$  with the  hyperbolic metric $g_{hyp}$ which is polarized
with respect to the anticanonical bundle.
Then, being $(\Sigma, g_{hyp})$ locally homegeneous all the coefficents $a_j$ of TYZ expansion
are constant,  more precisely $a_1$ is half of the constant scalar curvature and one can show that  $a_k=0$, for $k\geq 2$.
On the other hand, as we have already pointed out, $mg_{hyp}$ is not projectively induced (see \cite{loiherm} for a proof).

Finally, notice that   prescribing the values of  the  coefficients
of the  TYZ expansion gives rise to interesting elliptic PDE
as shown by Z. Lu and G. Tian  \cite{logterm}. The main result obtained there  is that if the log term
of the Szeg\"{o} kernel  of the unit disk bundle over $M$ vanishes then  $a_k=0$, for  $k>n$.
Hence, in the light of  the previous lemma and considerations, we  believe that the link between the Szeg\"{o} kernel and our results deserves further study.

\section{The proof of  Proposition   \ref{mainteor1}}\label{secmainteor1}
In this section  we assume that $L$ is a polarization of a compact complex manifold $M$ which admits a non-positively curved \K--Einstein metric $g$
in $c_1(L)$.
Obviously if $c_1(M)<0$ we can take $L=K$, where $K$ is the canonical bundle over $M$, while  when $c_1(M)=0$  the polarization could not exist, take for example  a complex torus which is not an abelian variety. Moreover,  in both cases the existence of a \K--Einstein metric with negative or zero scalar curvature is guaranteed by Yau's solution  of Calabi's conjecture.
Notice also  that in both cases the manifold $(M, L)$ is asymptotically Chow polystable. Indeed, when $c_1(M)<0$, $\Aut (M)$ is finite and hence
the assertion follows by Donaldson's  Theorem \ref{thmdo} above. On the other hand, if $c_1(M)=0$, it is well-known that the set $h_0(M)$ of holomorphic fields on $M$ with zeros is trivial. Since the Lie algebra of the identity component of $\Aut(M, L)$ is exactly $h_0(M)$ (see, for example, \cite{Ga}, Prop. 7.1.2), we conclude by applying Theorem \ref{thmdo} again.

Therefore the set ${\mathcal B}_c(L)$ defined in the previous section is infinite. On the other hand the metric $m g$ is not projectively induced for any
$m$ as it follows by a Theorem of D. Hulin \cite{hu} which asserts that the scalar curvature of a projectively induced \K\--Einstein  metric is  strictly positive.
Combining this fact with  Lemma \ref{constcoeff2} the proof of Proposition \ref{mainteor1} is immediate.

\section{\K-Einstein metrics on low-dimensional toric manifolds}\label{secmainteor2}

Let us briefly recall that a compact, complex manifold $M$ of complex dimension $n$ is said to be {\it toric} if it contains a complex torus $({\C}^*)^n$ as a dense open subset, together with a holomorphic action $({\C}^*)^n \times M \rightarrow M$ that extends the natural action of $({\C}^*)^n$ on itself. A basic fact in the theory of toric manifolds is that any such $M$ is determined by the combinatorial data encoded in a {\it fan of cones} in ${\R}^n$, that is a set of convex linear cones satisfying some properties which the interested reader can find, for example, in \cite{Fu}. Moreover, any ample linear bundle $L$ on $M$ corresponds to a polytope $\Delta_L = \{x \in {\R}^n \ | \ \langle x, u_i \rangle \leq \lambda_i, \ \ i=1 , \dots , d \}$, where $u_i$, $i=1, \dots, d$, are integral vectors which generate the edges of the cones in the fan and $\lambda_i \in {\Z}$. In particular, when $M$ is Fano the anticanonical bundle $K^*$ corresponds to the choice $\lambda_i = 1$, $i = 1, \dots, d$. This correspondence is such that in the coordinates $z = (z_1, \dots, z_n) \in ({\C}^*)^n$ defined on the open dense subset of $M$ diffeomorphic to the complex torus, a basis of the space $H^0(L)$ of global sections of $L$ is given by $S = \{ z^{J_0}, \dots , z^{J_N} \}$, where $\{J_0, \dots, J_N\} = \Delta_L \cap {\Z}^n$ and, for any $J = (j_1, \dots, j_n) \in {\Z}^n$, we set $z^J = z_1^{j_1} \cdots z_n^{j_n}$ .

 Since the fan associated to a toric manifold is determined up to the action of $SL(n, {\Z})$, we can always assume that it contains the $n$-dimensional cone generated by $-e_1, \dots, -e_n$, where  $\{e_1, \dots, e_n\}$ is the canonical basis of ${\R}^n$. Moreover, it is known (\cite{Au}) that polytopes which differ by a translation via a vector $v \in {\Z}^n$ represent isomorphic line bundles on the same toric manifold. Then, it follows that the anticanonical bundle can be represented by a polytope $\Delta$ which satisfies the following

\begin{assumption}\label{assumption}
$\Delta$ contains the origin $(0, \dots, 0)$ as vertex and the edge at this vertex is generated by $+e_1, \dots, +e_n$.  \end{assumption}

\begin{lemma}\label{lemma}
Under Assumption \ref{assumption}, a projectively induced toric metric $\omega \in c_1(K^*)$ writes in the coordinates $z_1, \dots, z_n$ on the open dense subset diffeomorphic to $({\C}^*)^n$ as $\omega = \frac{i}{2} \partial \bar \partial \log F$ where $F = \sum_I a_{I} x^I$ is a polynomial in $x = (x_1 = |z_1|^2, \dots, x_n = |z_n|^2)$ such that $a_{I} \geq 0$ and $a_{I} >0$ if and only if $I \in \Delta \cap {\Z}^n$.

\noindent Moreover, $\omega$ is \K-Einstein if and only if $F$ satisfies the following polynomial equality
\begin{equation}\label{detDIMn}
\det(A) = c F^{2n - 1}, \ \ \ A_{ij} = \left( F F_{ij} - F_i F_j \right) \bar z_i z_j + F F_j \delta_{ij}, \ \ \ c \in {\R}^+
\end{equation}
where $F_i = \frac{\partial F}{\partial x_i}$, $F_{ij} = \frac{\partial^2 F}{\partial x_i \partial x_j}$, \ $i,j=1, \dots, n$.
\end{lemma}

\begin{proof}
In the coordinates $z_1, \dots, z_n$, a generic element $v_i$ of a basis of $H^0(K^*)$ writes $\sum_k f_{ik} z^{J_k}$, where $\{J_0, \dots, J_N\} = \Delta \cap {\Z}^n$ and $(f_{jk}) \in GL(N+1, {\C})$, so a projectively induced toric metric (i.e. invariant by the action of the real torus ${\T}^n = \{(e^{i \theta_1}, \dots, e^{i \theta_n})\}$ on $M$ given on $({\C}^*)^n$ by
$$(z_1, \dots, z_n) \mapsto (e^{i \theta_1} z_1, \dots , e^{i \theta_n} z_n) \ )$$
 can be written as
$$\omega =\frac{i}{2}   \partial \bar \partial \log F(x_1, \dots, x_n),$$
 where $x_i = |z_i|^2$ and $F$ is a linear combination with real coefficients of the functions $x^I = x_1^{i_1} \cdots x_n^{i_n}$, where $I= (i_1, \dots, i_n) \in \Delta \cap {\Z}^n$. Notice that, since $(f_{jk}) \in GL(N+1, {\C})$, we have that the coefficient of $x^{J_l}$ in $F$ is given by $\sum_i |f_{il}|^2 > 0$. By Assumption \ref{assumption} and the convexity of $\Delta$ it follows that $\Delta \subseteq \bigcap \{ x_i \geq 0 \}$, so $F$ is in fact a polynomial. Notice that $\{e_1, \dots , e_n \} \subseteq \Delta \cap {\Z}^n$, so that $F = 1 + \alpha_1 x_1 + \cdots + \alpha_n x_n +$ (terms of higher order), with $\alpha_i >0$ for each $i = 1, \dots, n$.

Now, observe that the matrix $g_{\omega}$ of the metric associated to $\omega$ has entries $g_{i\bar j} = \frac{(F F_{ij} - F_{i} F_{j}) \bar z_i z_j + F F_j \delta_{ij}}{F^2}$, so that $\det(g_{\omega}) = \frac{\det(A)}{F^{2n}}$, where $A_{ij}$ is given by (\ref{detDIMn}).
By the well-known formula for the Ricci form $\rho_{\omega} = - i  \partial \bar \partial \log \det(g_{\omega})$ we then see that the Einstein condition is equivalent to the equation $\partial \bar \partial \Phi = 0$, where $\Phi = \log \left( \frac{\det(A)}{F^{2n-1}} \right)$.
Now, since $\Phi = \Phi(x_1, \dots , x_n)$, this is equivalent to the equations
$$\frac{\partial \Phi}{\partial x_i} + \frac{\partial ^2 \Phi}{\partial x_i^2} x_i = 0, \ i=1, \dots, n, \ \ \frac{\partial ^2 \Phi}{\partial x_i \partial x_j} = 0, \ i \neq j .$$
The last set of equations implies that $\frac{\partial \Phi}{\partial x_i}$ depends only on $x_i$, for every $i=1, \dots, n$, so that the first $n$ equations become ordinary differential equations whose general solution is $\frac{\partial \Phi}{\partial x_i} = c_i/x_i$, for $c_i \in {\R}$, and $\Phi = c_i \log x_i + f(x_1, \dots, x_{i-1}, x_{i+1}, \dots, x_n)$ for some real function $f$. Since this holds true for every $i = 1, \dots, n$, one concludes that
\begin{equation}\label{Philog}
\Phi = c_1 \log x_1 + \dots + c_n \log x_n + c_{n+1} = \log(x_1^{c_1} \cdots x_n^{c_n}) + c_{n+1} ,
\end{equation}
for $c_i \in {\R}$. We claim that $c_i = 0$ for $i \leq n$.
Indeed, since $F = 1 + \alpha_1 x_1 + \cdots + \alpha_n x_n +$ (terms of higher order), $\alpha_i > 0$, by $\Phi = \log \left( \frac{\det(A)}{F^{2n-1}} \right)$ and by noticing that at $z_1 = \dots = z_n = 0$ we have $A = diag(F F_1, \dots, F F_n)$, we see that when $x_1, \dots , x_n \rightarrow 0$ then $\Phi \rightarrow \log (\alpha_{1} \cdots \alpha_{n}) \in {\R}$. On the other hand, by (\ref{Philog}) one easily checks that $\Phi \rightarrow \pm \infty$ for $x_1, \dots, x_n \rightarrow 0$ along suitable curves in $({\R}^+)^n$ if at least one among the $c_i$'s, $i=1, \dots, n$, does not vanish (for example, if $c_i \neq 0$ take $x_i = t^k, x_j = t$ for $j \neq i$, for $k$ large enough). This proves the claim and then (\ref{detDIMn}), for $c = e^{c_{n+1}}$.
\end{proof}

We are now ready to prove the following  proposition interesting on its own sake.

\begin{prop}\label{corollary}
If $M$ is a smooth, compact toric $n$-dimensional  manifold, $n \leq 4$, then $M$ does not admit any projectively induced \K-Einstein metric unless it is a projective space or the  product of projective spaces.
\end{prop}
\begin{proof}
Suppose, contrary to the claim, that there exists a projectively induced \K-Einstein metric $\omega$ on $M$. Let us first assume that $\omega \in c_1(K^*)$.

Let us recall (\cite{BS}) that, for $n \leq 4$, the toric manifolds $M$ which admit a \K-Einstein metric are completely classified. More precisely, if $n = 2$, $M$ is either ${\C} {\Proj}^2$, ${\C} {\Proj}^1 \times {\C} {\Proj}^1$ or the blow-up of ${\C} {\Proj}^2$ at three points. This last case is associated to the fan in ${\R}^2$ whose set of edges is generated by
\begin{equation}\label{dim2}
\pm e_1, \ \pm e_2, \ \pm(e_1 - e_2).
\end{equation}
If $n=3$, either $M$ is ${\C} {\Proj}^3$, or can be decomposed into a product of lower dimensional manifolds or it is the manifold associated to the fan in ${\R}^3$ whose cones have the following generators:

\begin{equation}\label{dim3 1}
e_1, \ e_2, \ \pm e_3, \ -(e_1 + e_3), \ -(e_2 - e_3).
\end{equation}

Finally, if $n=4$, either $M$ is ${\C} {\Proj}^4$, or can be decomposed into a product of lower dimensional manifolds or it is the manifold associated to one of the following fans in ${\R}^4$ (given by the generators of their cones):
\begin{equation}\label{dim4 1}
\pm e_1, \ \dots, \ \pm e_4, \ \pm (e_1 + \cdots + e_4) \ ;
\end{equation}
\begin{equation}\label{dim4 2}
e_1, \ e_2, \ \pm e_3, \ \pm e_4, \ \pm (e_3 + e_4), \ -(e_1 - e_3), \ -(e_2 + e_3) \ ;
\end{equation}
\begin{equation}\label{dim4 3}
e_1, \ \dots, \ e_4, \ -(e_1 + \cdots + e_4), \ -(e_1 + e_2), \ -(e_3 + e_4), \ e_1 + e_3, \ e_2 + e_4 .
\end{equation}
In order to prove the Proposition, we first apply to each of the above fans a suitable transformation $A \in SL(n, {\Z})$ so that the anticanonical bundle of the corresponding manifold can be represented by a polytope $\Delta$ satisfying the properties given in Assumption \ref{assumption} and Lemma \ref{lemma} can be applied. The case (\ref{dim2}) already meets the required condition, while one easily verifies that the following matrices

\begin{displaymath} - I_3, \ \ \left( \begin{array}{cccc}
1 & 0 & \ 0 & \ 0\\
0 & 1 & \ 0 &\  0\\
0 & 0 & -1 &\ 0\\
0 & 0 & \ 0 & -1\\
\end{array} \right), \ \ \left( \begin{array}{cccc}
-1 &\  0 & 0 &\  0\\
\ 0 & -1 & 0 &\  0\\
\ 0 &\  0 & -1 &\  1\\
\ 0 & \ 0 & 0 & -1\\
\end{array} \right), \ \ \left( \begin{array}{cccc}
-1 & \ \ 0 &\  1 &\  0\\
\ 0 & -1 &\  0 &\  1\\
\ 0 &\  \ 0 & -1 &\  0\\
\ 0 &\  \ 0 &\  0 & -1\\
\end{array} \right)
\end{displaymath}

\noindent have this property, respectively for (\ref{dim3 1}), (\ref{dim4 1}), (\ref{dim4 2}), (\ref{dim4 3}). Then the anticanonical bundles are represented respectively by the polytopes $\Delta = $

\begin{equation}\label{politopoDim2}
\{ (x,y) \in {\R}^2 \ | \ 0 \leq x, y \leq 2, \ -1 \leq x - y \leq 1 \} ;
\end{equation}
\begin{equation}\label{politopoDim3}
\{ (x,y,z) \in {\R}^3 \ | \ x, y \geq 0, \ 0 \leq z \leq 2, \ x+z \leq 3, \ y-z \leq 1 \} ;
\end{equation}
\begin{equation}\label{politopoDim41}
\{ (x,y,z,w) \in {\R}^4 \ | \ 0 \leq x, y, z, w \leq 2, \ -1 \leq x + y - z - w \leq 1 \} ;
\end{equation}
\begin{equation}\label{politopoDim42}
\{ (x,y,z,w) \in {\R}^4 \ | \ x, y \geq 0, \ 0 \leq z, w \leq 2, \ -1 \leq z - w \leq 1, \ x-z \leq 1, \ y+z \leq 3 \};
\end{equation}
\begin{equation}\label{politopoDim43}
\begin{split}
\{ (x,y,z,w) \in {\R}^4 \ | \ x, y, z, w \geq 0,  \ x - z \leq 1, \ y - w \leq 1, \ z + w \leq 3, \  x + y \leq 1, \\ z + w - x - y \leq 1 \} .
\end{split}
\end{equation}

We are now going to treat each of the cases (\ref{politopoDim2})-(\ref{politopoDim43}) above separately. Let $F$ be the polynomial given in Lemma \ref{lemma}. We shall use the following notation:
$$F_{x_1^{i_1} \dots x_n^{i_n}} = \frac{\partial ^{i_1 + \cdots + i_n} F}{\partial x_1^{i_1} \dots \partial x_n^{i_n}}|_{x_1 = \cdots = x_n = 0}, \ \ \ \tilde F_{x_1^{i_1} \dots x_n^{i_n}} = \frac{F_{x_1^{i_1} \dots x_n^{i_n}}}{(F_{x_1})^{i_1} \cdots (F_{x_n})^{i_n}} .$$
In order to prove the Proposition, we check that equation (\ref{detDIMn}) implies in each of the cases (\ref{politopoDim2})-(\ref{politopoDim43}) that $\tilde F_{x_1^{i_1} \dots x_n^{i_n}} = 0$ for some $(i_1, \dots, i_n) \in \Delta \cap {\Z}^n$, against the first part of the statement of Lemma \ref{lemma}.

\medskip

\noindent
Case (\ref{politopoDim2}): the left-hand side $\det(A)$ in equation (\ref{detDIMn}) can be written more explicitly as
\begin{equation}\label{detindimens2}
\begin{split}
[(F_{11} F - F_1^2)(F_{22} F - F_2^2) - (F_{12}F - F_1 F_2)^2] x_1 x_2 + F_2 F (F_{11} F - F_1^2) x_1 +\\ +F_1 F (F_{22} F - F_1^2) x_2 + F_1 F_2 F^2 .
\end{split}
\end{equation}
By a straight calculation one gets the following equalities at $x_1=x_2=0$:
$$\det(A) = F_{x_1} F_{x_2}, \ \ \frac{\partial \det(A)}{\partial x_1} = F_{x_1}^2 F_{x_2} + F_{x_1} F_{x_1 x_2}, \ \ \frac{\partial \det(A)}{\partial x_2} = F_{x_2}^2 F_{x_1} + F_{x_2} F_{x_1 x_2},$$
$$\frac{\partial^2 \det(A)}{\partial {x_1}^2} = 2 F_{x_1}^2 F_{x_1 x_2} + F_{x_1^2 x_2} F_{x_1}, \ \ \frac{\partial^2 \det(A)}{\partial {x_2}^2} = 2 F_{x_2}^2 F_{x_1 x_2} + F_{x_1 x_2^2} F_{x_2} .$$
By comparing with the corresponding derivatives of the right-hand side of equation (\ref{detDIMn}), one gets $\tilde F_{x_1 x_2} = \tilde F_{x_1^2 x_2} = \tilde F_{x_1 x_2^2} = 2$. One then gets a contradiction by substituting these values into
$$\frac{\partial^2 \det(A)}{\partial x_1 \partial x_2}|_{x_1=x_2=0} = 4 F_{x_1} F_{x_2} F_{x_1 x_2} + 2 F_{x_2} F_{x_1^2 x_2} + 2 F_{x_1} F_{x_1 x_2^2} $$
and comparing with $\frac{\partial^2 (c F^3)}{\partial x_1 \partial x_2}|_{x_1=x_2=0}$.

\medskip

\noindent Case (\ref{politopoDim3}) : by calculating the derivatives of both sides of equation (\ref{detDIMn}) with respect to $x_2$ one gets the following system
\begin{displaymath}\left\{
\begin{array}{ll}\tilde F_{x_1 x_2} + \tilde F_{x_2 x_3} = 3 \\2 \tilde F_{x_1 x_2} \tilde F_{x_2 x_3} + \tilde F_{x_2^2 x_3} = 6 \\ \tilde F_{x_1 x_2} \tilde F_{x_2^2 x_3} = 2 \end{array} \right. \end{displaymath}
which has as unique solution $\tilde F_{x_1 x_2} = 1, \tilde F_{x_2 x_3} = 2, \tilde F_{x_2^2 x_3} = 2$.

By calculating the derivatives of both sides of equation (\ref{detDIMn}) with respect to $x_3$ one gets
\begin{small}
\begin{displaymath}\left\{
\begin{array}{ll} \tilde F_{x_1 x_3} + \tilde F_{x_2 x_3} + 2 \tilde F_{x_3^2} = 3 \\ \tilde F_{x_1 x_3^2} + 2 \tilde F_{x_1 x_3} \tilde F_{x_2 x_3} + \tilde F_{x_2^2 x_3} + 4 \tilde F_{x_3^2} (\tilde F_{x_1 x_3} + \tilde F_{x_2 x_3}) + \tilde F_{x_3^2}  = 6 + 3 \tilde F_{x_3^2} \\ \tilde F_{x_1 x_3^2} \tilde F_{x_2 x_3} + F_{x_2 x_3^2} \tilde F_{x_1 x_3}  + 2 \tilde F_{x_3^2} (\tilde F_{x_1 x_3^2} + 2 \tilde F_{x_1 x_3} \tilde F_{x_2 x_3} + \tilde F_{x_2^2 x_3}) + \tilde F_{x_3^2}(\tilde F_{x_1 x_3} + \tilde F_{x_2 x_3}) = 2 + 6 \tilde F_{x_3^2} \\ \tilde F_{x_1 x_3^2} \tilde F_{x_2 x_3^2} + 4\tilde F_{x_3^2} (F_{x_1 x_3^2} \tilde F_{x_2 x_3} + F_{x_2 x_3^2} \tilde F_{x_1 x_3}) + \tilde F_{x_3^2} (\tilde F_{x_1 x_3^2} + 2 \tilde F_{x_1 x_3} \tilde F_{x_2 x_3} + \tilde F_{x_2^2 x_3}) = 6 \tilde F_{x_3^2} + 3 \tilde F_{x_3^2}^2 \\ 2 \tilde F_{x_1 x_3^2} \tilde F_{x_2 x_3^2} + \tilde F_{x_1 x_3^2} \tilde F_{x_2 x_3} + \tilde F_{x_2^2 x_3} \tilde F_{x_1 x_3} = 3 \tilde F_{x_3^2} \\ \tilde F_{x_1 x_3^2} \tilde F_{x_2 x_3^2} = \tilde F_{x_3^2}^2 \end{array} \right. \end{displaymath} \end{small}
which, solved by substitution from the last equation, implies $\tilde F_{x_3^2} = \frac{1}{2}$ and then, from the first equation and by the previous system, $\tilde F_{x_1 x_3} = 0$, which contradicts $(1,0,1) \in \Delta \cap {\Z}^n$.

\medskip

\noindent
Case (\ref{politopoDim41}) : by calculating the derivatives of both sides of equation (\ref{detDIMn}) with respect to $x_4$ one gets the following system

\begin{displaymath}\left\{
\begin{array}{ll}\tilde F_{x_1 x_4^2} \tilde F_{x_2 x_4^2} = 4 \\ \tilde F_{x_2 x_4} \tilde F_{x_1 x_4^2} + \tilde F_{x_1 x_4} \tilde F_{x_2 x_4^2} = 8\\ \tilde F_{x_1 x_4^2} + 2 \tilde F_{x_1 x_4} \tilde F_{x_2 x_4} + \tilde F_{x_2 x_4^2} = 12 \\ \tilde F_{x_1 x_4} + \tilde F_{x_2 x_4} = 4 \end{array} \right. \end{displaymath}

\noindent which is seen to have as unique solution $\tilde F_{x_1 x_4} = \tilde F_{x_2 x_4} = \tilde F_{x_1 x_4^2} = \tilde F_{x_2 x_4^2} = 2$. By the symmetry of the equations defining the polytope, we also get
$\tilde F_{x_1 x_3} = \tilde F_{x_1 x_3^2} = \tilde F_{x_2 x_3} = \tilde F_{x_2 x_3^2} = \tilde F_{x_1^2 x_3} = \tilde F_{x_1^2 x_4} = \tilde F_{x_2^2 x_3} = \tilde F_{x_2^2 x_4} = 2$.

\noindent By considering the $\frac{\partial}{\partial x_1}$, $\frac{\partial^2}{\partial x_1 \partial x_2}$, $\frac{\partial^3}{\partial x_1^2 \partial x_2}$, $\frac{\partial^4}{\partial x_1^2 \partial x_2^2}$ derivatives of both sides of equation (\ref{detDIMn}) and taking into account the above data, we get the equations $\tilde F_{x_1 x_2 x_3} + \tilde F_{x_1 x_2 x_4} = 4$, $\tilde F_{x_1 x_2 x_3} \tilde F_{x_1 x_2 x_4} = 4$ which immediately yield $\tilde F_{x_1 x_2 x_3} = \tilde F_{x_1 x_2 x_4} = 2$. Again by the symmetry of the equations of the polytope, we have also $\tilde F_{x_1 x_3 x_4} = \tilde F_{x_2 x_3 x_4} = 2$.

\noindent Now, notice that $\Delta \cap \{ y=z=0 \}$ gives the $2$-dimensional polytope of Case (\ref{politopoDim2}). It follows that we can use the calculations made in that case in order to get the $\frac{\partial^2}{\partial x_1 \partial x_4}$ derivative of both sides of equation (\ref{detDIMn}), evaluated at $x_2=x_3=0$. By replacing in the result obtained the values found above, one easily gets a contradiction.

\medskip

\noindent Case (\ref{politopoDim42}) : by calculating the derivatives of both sides of equation (\ref{detDIMn}) with respect to $x_3$ one gets the following system

\begin{displaymath}\left\{
\begin{array}{ll}\tilde F_{x_1 x_3} \tilde F_{x_2 x_3} \tilde F_{x_3^2 x_4} = 2 \\ 2 \tilde F_{x_1 x_3} \tilde F_{x_2 x_3} \tilde F_{x_3 x_4} + (\tilde F_{x_1 x_3} + \tilde F_{x_2 x_3}) \tilde F_{x_3^2 x_4} = 8 \\ 2 \tilde F_{x_1 x_3} \tilde F_{x_2 x_3} + 2 \tilde F_{x_1 x_3} \tilde F_{x_3 x_4} + 2 \tilde F_{x_2 x_3} \tilde F_{x_3 x_4} + \tilde F_{x_3^2 x_4} = 12 \\ \tilde F_{x_1 x_3} + \tilde F_{x_2 x_3} + \tilde F_{x_3 x_4} = 4 \end{array} \right. \end{displaymath}

\noindent which is seen to have as unique solution $\tilde F_{x_1 x_3} = \tilde F_{x_2 x_3} = 1, \tilde F_{x_3 x_4} = \tilde F_{x_3^2 x_4} = 2$. By exchanging the role of $x_3$ and $x_4$, one finds an analogous system from which one similarly gets $\tilde F_{x_1 x_4} = \tilde F_{x_2 x_4} = 1, \tilde F_{x_3 x_4} = \tilde F_{x_3 x_4^2} = 2$. By substituting these data into the equations $\tilde F_{x_1 x_2} + \tilde F_{x_1 x_3} + \tilde F_{x_1 x_4} = 4$, $2 \tilde F_{x_2^2} + \tilde F_{x_1 x_2} + \tilde F_{x_2 x_3} + \tilde F_{x_2 x_4} = 4$ (which arise from equation (\ref{detDIMn}) derivated with respect to $x_1$ and $x_2$ respectively) one gets $\tilde F_{x_2^2} = 0$, which contradicts $(0,2,0,0) \in \Delta \cap {\Z}^n$.

\medskip

\noindent Case (\ref{politopoDim43}) : as in the previous case, by calculating the derivatives of both sides of equation (\ref{detDIMn}) with respect to $x_1$ one gets the system

\begin{displaymath}\left\{
\begin{array}{ll}\tilde F_{x_1 x_2} \tilde F_{x_1 x_4} \tilde F_{x_1^2 x_3} = 2 \\ 2 \tilde F_{x_1 x_2} \tilde F_{x_1 x_3} \tilde F_{x_1 x_4} + (\tilde F_{x_1 x_2} + \tilde F_{x_1 x_4}) \tilde F_{x_1^2 x_3} = 8 \\ 2 \tilde F_{x_1 x_2} \tilde F_{x_1 x_3} + 2 \tilde F_{x_1 x_2} \tilde F_{x_1 x_4} + 2 \tilde F_{x_1 x_4} \tilde F_{x_1 x_3} + \tilde F_{x_1^2 x_3} = 12 \\ \tilde F_{x_1 x_2} + \tilde F_{x_1 x_3} + \tilde F_{x_1 x_4} = 4 \end{array} \right. \end{displaymath}

\noindent which is seen as above to have as unique solution $\tilde F_{x_1 x_2} = \tilde F_{x_1 x_4} = 1, \tilde F_{x_1 x_3} = \tilde F_{x_1^2 x_3} = 2$. By exchanging the roles of $x_3$ and $x_4$ and of $x_1$ and $x_2$ one finds an analogous system from which one similarly gets $\tilde F_{x_2 x_3} = \tilde F_{x_1 x_2} = 1, \tilde F_{x_2 x_4} = \tilde F_{x_2^2 x_4} = 2$. By substituting these data into the equation $\tilde F_{x_1 x_3} + \tilde F_{x_2 x_3} = 4$ (which arises from equation (\ref{detDIMn}) derivated with respect to $x_3$) one gets a contradiction.

\medskip

In order to conclude the proof of Proposition \ref{corollary}, we need to consider the general case when the embedding inducing the \K-Einstein metric is given by a basis of the space $H^0((K^*)^{\alpha})$ of the global sections of a power $(K^*)^{\alpha}$ of the anticanonical bundle $K^*$.
From the theory of toric manifolds it is known that a polytope representing $(K^*)^{\alpha}$ is obtained from the polytope $\Delta$ representing $K^*$ by applying the homothety of constant $\alpha$. Now, in each of the cases (\ref{politopoDim2})-(\ref{politopoDim43}) one sees that $\Delta$ has $e_i$ as vertex, for some $i = 1, \dots, n$, so $\alpha \Delta$ is an integral polytope if and only if $\alpha$ is a positive integer (this means that $K^*$ is indivisible).
As above, we have that any projectively induced \K-Einstein metric representing $c_1((K^*)^{\alpha})$ writes in the coordinates $z_1, \dots , z_n$ as $\frac{i}{2}  \partial \bar \partial G$, where $G = G(x_1, \dots , x_n)$ is now a polynomial with positive coefficients in the monomials $x_1^{i_1} \cdots x_n^{i_n}$, $(i_1, \dots, i_n) \in \alpha \Delta \cap {\Z}^n$.
Since the Einstein equation reads now $\rho_{\omega} = \frac{2}{\alpha} \omega$, arguing as above one gets the following equation (with notations as above)
\begin{equation}\label{detDIMn2}
\det(A) = c G^{2n - \frac{1}{\alpha}}, \ \ \ A_{ij} = (G G_{ij} - G_{i} G_{j}) \bar z_i z_j + G G_j \delta_{ij}
\end{equation}
which, $G$ being a polynomial, cannot be true unless $G = F^{\alpha}$ for some polynomial $F = F(x_1, \dots, x_n)$.
So $\frac{i}{2}  \partial \bar \partial G = \alpha \frac{i}{2}  \partial \bar \partial F$, and $\frac{i}{2}  \partial \bar \partial F$ is a projectively induced, \K-Einstein metric belonging to $c_1(K^*)$, which contradicts what we have obtained in the first part of the proof.
\end{proof}

We can now  prove  Theorem \ref{mainteor2}.

\begin{proof}[Proof of Theorem \ref{mainteor2}]
It is known that the polarized manifold $(M, L)$ is asymptotically Chow polystable and hence the set
${\mathcal B}_c(L)$ has  infinite elements. For the reader's convenience,
let us outline a proof here.
By the Futaki's reformulation of the results proved in \cite{mab}, $(M, L)$ is asymptotically Chow polystable provided the Lie algebra characters $F_i: \Lie(\Aut(M)) \rightarrow \C$, $i = 1, \dots, n$, introduced in \cite{Fut}, vanish (see also \cite{FOS}). Now, the \K-Einstein toric $n$-dimensional manifolds, $n \leq 4$, are {\it symmetric}, i.e. the only character on the algebraic torus $G = (\C^*)^n \subseteq \Aut(M)$ which is invariant by the action by conjugation of the normalizer $N(G)$ of $G$ in $\Aut(M)$ is the trivial one (see, for example, \cite{BS}).
 One then sees that the restrictions of the $F_i$'s to the Lie algebra $\Lie(G)$ of $G$ must vanish. By Matsushima's results on \K-Einstein manifolds, $\Aut(M)$ is reductive so that $\Lie(\Aut(M)) = [\Lie(\Aut(M)), \Lie(\Aut(M))] + Z(\Lie(\Aut(M)))$, where $Z(\Lie(\Aut(M)))$ denotes the center of $\Lie(\Aut(M))$. So, by $Z(\Lie(\Aut(M))) \subseteq \Lie(G)$ ($G$ is a maximal torus) and the fact that the $F_i$'s are Lie algebra characters, we see that the $F_i$'s vanish on all $\Lie(\Aut(M))$ and this proves the claim by Futaki's above-mentioned result.

\noindent Now, let $g_B\in {\mathcal B}(L)$.  If $M$ is either a projective space or the product of  projective spaces then it is not hard to see that   the cardinality of ${\mathcal B}_{g_B}$ is infinite (see  Section 4 in \cite{arlcomm} for a  proof). Otherwise,  it follows by
Proposition \ref{corollary} and part (1) of Lemma \ref{constcoeff2} that   ${\mathcal B}_{g_B}$ consists of a finite numbers of
balanced metrics.
\end{proof}

\begin{remar}\rm
The main difficulty in extending this theorem to arbitrary dimensions is that there is not a classification of
toric manifolds for $n\geq 5$.
\end{remar}

\section{The constant scalar curvature   case}\label{secmainteor3}
Recall the following from \cite{AP} and \cite{AP2}:

\begin{theo}
\label{blow1}
Let $(M, g, \omega)$ be a constant scalar curvature \K\ metric on a compact $n$-dimensional complex
manifold.  Assume that
there is no nonzero holomorphic vector field vanishing somewhere on
$M$. Then, given finitely many points $p_1, \ldots, p_k \in
M$ and positive numbers $b_1, \ldots, b_k >0$, there exists $\e_0
>0$ such that the blow up $\tilde M = Bl_{p_1,\dots ,p_k}M$ of $M$ at $p_1, \ldots, p_k$
carries constant scalar curvature \K\ forms
\[
\omega_\e \in \pi^{*} \, [\omega] - \e ^{2} \, (b_1 \, PD[E_1] +
\ldots + b_k \, PD[E_k]),
\]
where the $PD[E_j]$ are the Poincar\'e duals of the $(2n-2)$-homology
classes of the exceptional divisors of the blow up at $p_j$ and $\e
\in (0,\e_0)$.
Morever, the sequence of metrics $g_\e$ converges to $g$ in ${\mathcal C}^{\infty} (M \setminus \{p_1, \ldots, p_k\})$. \label{th:2}

\medskip

If the scalar curvature of $g$ is not zero then the scalar
curvatures of $g_\e$ (the metric associated to $\omega_{\e}$) and  of $g$ have the same signs.
\label{th:1.1}
\end{theo}

Let us denote by ${\mathfrak h}$ the space of  hamiltonian holomorphic vector fields and by
\[
\xi_\omega : M  \longmapsto  {\mathfrak h}^*
\]
the {\em  moment} map which is defined by requiring that if $\Xi \in {\mathfrak h}$, the function $ \zeta_{\omega} : = \langle \xi_\omega  , \Xi \rangle$ is a (complex valued) Hamiltonian for the vector field $\Xi$, namely the unique solution of
\[
- \bar \del \zeta _{\omega}  = \frac{_1}{^2} \, \omega (\Xi , -) ,
\]
which is normalized by
\[
\int_{M} \, \zeta_{\omega}  \, dvol_g =0 .
\]

With these notations, the result  obtained in \cite{AP2} reads~:
\begin{theo}
\label{blow2}
Assume that $(M, g, \omega)$ is a cscK compact complex manifold and that $p_1, \ldots, p_k \in M$ and $b_1, \ldots, b_k >0$ are chosen so that~:

\medskip

\begin{itemize}
\item[(i)] $\xi_\omega (p_1) , \ldots, \xi_\omega (p_k)$ span ${\mathfrak h}^*$ \\[3mm]
\item[(ii)] $\sum_{j=1}^k b_j^{n-1}  \, \xi_\omega (p_j)  = 0 \in {\mathfrak h}^* $ .
\end{itemize}

\medskip

Then there exists $\e_0 > 0$ such that, for all $\e \in (0, \e_0)$, there exists on $\tilde M = Bl_{p_1,\dots ,p_k}M$  a cscK metric $g_\e$ associated to the \K\ form
\[
\omega_\e \in \pi^{*} \, [\omega] - \e ^{2} \, (b_{1, \e} \, PD[E_1]
+ \ldots + b_{k, \e} \, PD[E_k]),
\]
where
\begin{equation}
|b_{j ,\e} - b_j | \leq c \, \e^{\frac{2}{2n+1}} .
\label{disto}
\end{equation}
Morever, the sequence of metrics $g_\e$ converges to $g$ in ${\mathcal C}^{\infty} (M \setminus \{p_1, \ldots, p_k\})$. \label{th:2}
\end{theo}
Therefore, in the presence of nontrivial hamiltonian holomorphic vector fields, the number of points which can be blown up, their position, as well as the possible \K\ classes on the blown up manifold have to satisfy some constraints.

Despite the fact we do not know explicitely the metrics given by the above
constructions, we want to show, that, at least for $\e$ sufficiently small, they cannot have
the second coefficient  $a_2$ in the TYZ expansion   equal to a constant.
\begin{prop}\label{m=2}
For  any family of metrics $g_{\e}$ constructed either by Theorem \ref{blow1} or Theorem \ref{blow2},
there exists $\e_1 >0$ such that the coefficent  $a_2(g_{\e})$ in TYZ expansion (see (\ref{coefflu}) above)  is not constant for $\e < \e_1$.
\end{prop}

\begin{proof}
We first argue by contradiction on $a_2$: suppose there exists a sequence $\e_j \rightarrow 0$ s.t.
$a_2(g_{\e_{j}})$ are indeed constants. Being cscK metrics this is equivalent to say that $|R_{g_{\e_{j}}}|^2 - 4|Ric_{g_{\e_{j}}}|^2 = C_{\e_{j}}$, for some constant $C_{\e_j}$. Scaling the metrics $g_{\e}$ around a point on any of the exceptional divisors by a factor $\frac{1}{\e^2}$, we know that $\frac{1}{\e^2}g_{\e}$ converges smoothly (as $\e \rightarrow 0$) to
the LeBrun-Simanca $g_{lbs}$ metric on $Bl_0 \C ^n$. In particular we would have
$|R_{g_{lbs}}|^2 - 4|Ric_{g_{lbs}}|^2 = lim_{\e_{j}\rightarrow 0} \e_j^4 C_{\e_{j}}$.
This would imply that $|R_{g_{lbs}}|^2 - 4|Ric_{g_{lbs}}|^2$ is constant on
$Bl_0 \C ^n$, contradicting  the following lemma.
\end{proof}

\begin{lemma}
Having called $(v_1, \dots, v_n)$ the standard euclidean coordinates on $Bl_0 \C ^n \setminus K$, for some compact $K$ containing the exceptional divisor, then
\begin{equation}\label{a2m=2}
(|R_{g_{lbs}}|^2 - 4|Ric_{g_{lbs}}|^2) (v_1 , 0) = - \frac{113}{|v_1|^8} + o(|v_1|^{-8} )
\end{equation}
 for $n=2$ as $v_1$ goes to infinity, and
 \begin{equation} \label{a2m}
 (|R_{g_{lbs}}|^2 - 4|Ric_{g_{lbs}}|^2) (v_1,0, \dots ,  0) = - \frac{a}{|v_1|^{4n}} + o(|v_1|^{-4n})
 \end{equation}
with $a>0$,  for $n \geq 3$ as $v_1$ goes to infinity.
\end{lemma}
\begin{proof}
Recall the following properties of the $g_{lbs}$: by construction, the \K\ form $\omega_{lbs}$ is
invariant under the action of $U(n)$. If $v = (v_{1}, \ldots, v_n)$
are complex coordinates in ${\mathbb C}^n -\{ 0 \}$, the \K\ form
$\omega_{lbs}$ can be written as
\begin{equation}
\omega_{lbs}  : =\frac{i}{2}   \del \bar \del  \left(\frac{_1}{^{2}} |v|^{2}
+E_n (|v|) \right) \label{etage}
\end{equation}
More precisely
\begin{equation}
\omega_{lbs}  =\frac{i}{2}   \del \bar \del   \left(\frac{_1}{^{2}} |v|^{2} +
\log |v| \right) \label{eta2} \end{equation} in dimension $n=2$. In
dimension $n \geq 3$, even though there is no explicit formula, we
have
\begin{equation}
\omega_{lbs}  = \frac{i}{2}   \del \bar \del   \left( \frac{_1}{^{2}} |v|^{2}
-
 |v|^{4-2n} + o (|v|^{2-2n}) \right)
\label{etam}
\end{equation}

Having in complex dimension $2$ an explicit \K\ potential, the proof of the theorem reduces to estimating the relevant quantities. For this purpose we get an explicit formula for the matrix $G_{lbs}$ which represents the metric $g_{lsb}$, namely:

$$G_{lbs} = \frac{1}{|v|^4}
\begin{pmatrix}
1+|v_1|^2+|v|^4 & \bar{v_1} v_2 \\
v_1 \bar{v_2} & 1+|v_2|^2+|v|^4
\end{pmatrix}.
$$
Observe that $\det(\omega_{lbs}) = 1 + \frac{1}{|v|^2} + {\mbox{higher order terms}}$, as $|v|$ goes to infinity.
From now on we will write $f \simeq h$ for two functions (or two matrixes) which agree up to higher order terms as $|v|$ goes to infinity.

Clearly $G_{lbs} ^{-1} \simeq Id$.  A straightforward computation now shows that
$$\frac{\partial g_{i\bar{l}}}{\partial v_k }  = \bar{v}_l |v|^{-4}\delta_{ik} -2|v|^{-6}\bar{v}_k(\delta_{il} + v_i\bar{v}_l) \,\, $$
and
\begin{equation*}
\begin{split}\frac{\partial^2 g_{i\bar{l}}}{\partial v_k \partial \bar{v}_j}=
 |v|^{-4}\delta_{lj}\delta_{ik} - 2 |v|^{-6}(\delta_{il}\delta_{kj} + \bar{v}_lv_j\delta_{ik} + \bar{v}_kv_i\delta_{lj} + \bar{v_l}v_i\delta_{kj})+\\
 + 6 |v|^{-8}(\bar{v}_kv_j\delta_{il} + \bar{v}_l\bar{v}_kv_iv_j) \,\, .
 \end{split}
 \end{equation*}

Hence, restricting on the complex line $v_2=0$, we get that the only non zero
contributions among the second derivatives of the metric are given by

$$\frac{\partial^2 g_{1\bar{1}}}{\partial v_1 \partial \bar{v}_1} = -\frac{5}{|v_1|^4},$$
$$\frac{\partial^2 g_{2\bar{1}}}{\partial v_1 \partial \bar{v}_2} = -\frac{1}{|v_1|^4},$$
$$\frac{\partial^2 g_{1\bar{1}}}{\partial v_2 \partial \bar{v}_2} = -\frac{2}{|v_1|^4},$$
$$\frac{\partial^2 g_{2\bar{2}}}{\partial v_2 \partial \bar{v}_2} = \frac{1}{|v_1|^4},$$

Now recall that
\begin{equation}
\label{Riem}
|R|^2 = \sum_{i,j,k,l,p,q,r,s =1}^{n} \overline{g^{i\bar{p}}}g^{j\bar{q}}\overline{g^{k\bar{r}}}g^{l\bar{s}} R_{i\bar{j}k\bar{l}}\overline{R_{p\bar{q}r\bar{s}}}
\end{equation}
and
$$R_{i\bar{j}k\bar{l}} = -\frac{\partial^2 g_{i\bar{l}}}{\partial v_k \partial \bar{v}_j} +
\sum_{p,q=1}^n g^{p\bar{q}}\frac{\partial g_{i\bar{p}}}{\partial v_k} \frac{\partial g_{q\bar{l}}}{\partial \bar{v}_j} .$$

This readily implies that the lowest order terms of the contributions of the Riemann tensor are
$$R_{1\bar{1}1\bar{1}} \simeq    \frac{5}{|v_1|^4} \,\, ,  $$
$$R_{1\bar{2}2\bar{1}} \simeq    \frac{2}{|v_1|^4} \,\, ,  $$
$$R_{1\bar{2}1\bar{2}} \simeq    \frac{1}{|v_1|^4} \,\, ,  $$
$$R_{2\bar{2}2\bar{2}} \simeq    -\frac{1}{|v_1|^4} \,\, ,  $$
hence $|R|^2 \simeq \frac{31}{|v_1|^8}$ on the line $v_2=0$.
Recall that
\begin{equation}
\label{Ric}
|Ric|^2 = \sum_{i,j,k,l =1}^{n} \overline{g^{i\bar{k}}}g^{j\bar{l}}Ric_{i\bar{j}}\overline{Ric_{k\bar{l}}}\qquad \mbox{and}
\end{equation}
$$Ric_{k\bar{p}} = g^{i\bar{l}}R_{i\bar{l}k\bar{p}}.$$
Hence, restricting on the complex line $v_2=0$, we get
$$|Ric|^2 \simeq Ric_{1\bar{1}}\overline{Ric_{1\bar{1}}} +  Ric_{2\bar{2}}\overline{Ric_{2\bar{2}}} \simeq \frac{36}{|v_1|^8}.$$
In summary we have proved that
$$a_2(g_{lbs}) _{|_{v_{2}=0}}= \left[|R|^2 -4|Ric|^2\right] _{|_{v_{2}=0}} \simeq - \frac{113}{|v_1|^8}$$
which is exactly  (\ref{a2m=2}).

Not having an explicit expression for the LeBrun-Simanca metric when $n>2$ we can give only estimates instead of precise formulae, though the line of the argument will be the same as in complex dimension $2$.

This time
$$\omega_{lbs} \simeq (1+|v|^{2-2n} \delta_{ij} + v_i\bar{v_j} |v|^{-2n})dv_i\wedge d\bar{v_j},$$
$$\frac{\partial g_{i\bar{p}}}{\partial v_k }  \simeq (1-n)\bar{v}_k |v|^{-2n}\delta_{ip}
+ |v|^{-2n}\bar{v}_p\delta_{ik} -n v_i\bar{v}_p\bar{v}_k|v|^{-2-2n} \,\, $$
and\\

$\frac{\partial^2 g_{i\bar{p}}}{\partial v_k \partial \bar{v}_j} \simeq
|v|^{-2n}[(1-n)\delta_{kj}\delta_{il}+ \delta_{jl}\delta_{ik}]
+
 |v|^{-2-2n}$
$[-n(1-n)v_j\bar{v}_k\delta_{il} - n\bar{v}_lv_j\delta_{ik}-nv_i\bar{v}_k\delta_{jl}-nv_i\bar{v}_l\delta_{kj}] + |v|^{-4-2n}[n(1+n)v_iv_j\bar{v}_l\bar{v}_k] .$\\

Hence, restricting on the complex line $v_2=\dots = v_n=0$, we get that the only non zero
contributions among the second derivatives of the metric are given by

$$\frac{\partial^2 g_{1\bar{1}}}{\partial v_1 \partial \bar{v}_1} \simeq
\frac{2(n-1)^2}{|v_1|^{2n}},$$
$$\frac{\partial^2 g_{1\bar{r}}}{\partial v_1 \partial \bar{v}_r} \simeq
\frac{1-n}{|v_1|^{2n}},$$
$$\frac{\partial^2 g_{1\bar{1}}}{\partial v_r \partial \bar{v}_r}  \simeq
\frac{1-2n}{|v_1|^{2n}},$$
$$\frac{\partial^2 g_{r\bar{r}}}{\partial v_1 \partial \bar{v}_1}  \simeq
\frac{(1-n)^2}{|v_1|^{2n}},$$
$$\frac{\partial^2 g_{r\bar{r}}}{\partial v_r \partial \bar{v}_r}  \simeq
\frac{2-n}{|v_1|^{2n}},$$
where $r$ is an index ranging from $2$ to $n$.
This immediately gives, on the line $v_1$,
$$|R|^2 \simeq |v_1|^{-4n}[4(n-1)^4 + (n-1)(1-2n)^2 + (1-n)^2 + (1-n)^2 + (2-n)^2] ,$$
and
$$|Ric|^2 \simeq |v_1|^{-4n}\left[-4(n-1)^4 + (n-1)(2-n)^2\right]$$
hence

\noindent
$|R|^2-4|Ric|^2   \simeq$

\noindent
$\simeq |v_1|^{-4n}\{(n-1) [-12(n-1)^3 + (1-2n)^2 +2(1-n)-4(2-n)^2 ] + (2-n)^2$
which is readily seen to be negative for $n\geq 3$. This ends the proof of claim  (\ref{a2m}) in all dmensions.
\end{proof}

We are now in the position to  prove Theorem \ref{mainteor3}. which in the previous notation reads:

\begin{theo}\label{mainteor3bis}
Let $\e = \frac{p}{q} < \e_1$ with $\e_1$ as in Proposition \ref{m=2} and let $L_{\e} \rightarrow M$ be a polarization for the \K\ class
of the metric $qg_{\e}$. Then,  for each   $g_B\in {\mathcal B}(L_{\e})$,  the set
${\mathcal B_{g_B}}$ is finite.
\end{theo}
\begin{proof}
The proof follows  by combining Proposition \ref{m=2} with part (2) of Lemma \ref{constcoeff2}.
\end{proof}

\begin{remar}\rm
We do not know if  in the previous theorem  ${\mathcal B}_{g_B}$  is empty.
 For example, in \cite{vezu} it is shown that there exist cscK polarizations $L$ on the blow up $M$ of $\C P^2$ at four points (all but one aligned) constructed from Theorem \ref{blow2} such that $(M, L^m)$ is not asymptotically Chow polystable for $m$ large enough, so that ${\mathcal B}_c(L)$ is finite.
One  could try to use part (1) instead of part (2) of Lemma \ref{constcoeff2} to prove Theorem \ref{mainteor3bis},
namely to show that $mg_{\e}$
is not projectively induced for all $m$.
\end{remar}

\end{document}